\documentclass[12pt,a4paper]{amsart}
\pdfoutput=1

\usepackage{amsmath,amssymb,amsthm}  
\usepackage{mathtools}
\usepackage{tikz}
\usepackage{thmtools, thm-restate}

\usepackage{xcolor} 	
\usepackage{hyperref}
\hypersetup{
	colorlinks,
    linkcolor={red!60!black},
    citecolor={green!60!black},
    urlcolor={blue!60!black},
}
\usepackage[abbrev, msc-links]{amsrefs} 
\usepackage{comment}

\usepackage[utf8]{inputenc}
\usepackage[T1]{fontenc}
\usepackage{lmodern}
\usepackage[babel]{microtype}
\usepackage[english]{babel}

\usepackage{geometry}
\usepackage{enumitem}

\let\polishlcross=\l
\def\l{\ifmmode\ell\else\polishlcross\fi}

\let\emptyset=\varnothing

\let\theta=\vartheta
\let\rho=\varrho
\let\phi=\varphi

\newcommand{\Set}[1]{{\left\lbrace {#1} \right\rbrace}}

\def\set#1:#2{\Set{{#1} \colon {#2}}}

\newcommand{\restrict}{\restriction}

\makeatletter
\def\moverlay{\mathpalette\mov@rlay}
\def\mov@rlay#1#2{\leavevmode\vtop{   \baselineskip\z@skip \lineskiplimit-\maxdimen
   \ialign{\hfil$\m@th#1##$\hfil\cr#2\crcr}}}
\newcommand{\charfusion}[3][\mathord]{
    #1{\ifx#1\mathop\vphantom{#2}\fi
        \mathpalette\mov@rlay{#2\cr#3}
      }
    \ifx#1\mathop\expandafter\displaylimits\fi}
\makeatother

\newcommand{\bigdcup}{\charfusion[\mathop]{\bigcup}{\cdot}}

\DeclareFontFamily{U}  {MnSymbolC}{}
\DeclareSymbolFont{MnSyC}         {U}  {MnSymbolC}{m}{n}
\DeclareFontShape{U}{MnSymbolC}{m}{n}{
    <-6>  MnSymbolC5
   <6-7>  MnSymbolC6
   <7-8>  MnSymbolC7
   <8-9>  MnSymbolC8
   <9-10> MnSymbolC9
  <10-12> MnSymbolC10
  <12->   MnSymbolC12}{}
\DeclareMathSymbol{\powerset}{\mathord}{MnSyC}{180}

\theoremstyle{plain}
\newtheorem{thm}{Theorem}[section]
\newtheorem{fact}[thm]{Fact}
\newtheorem{prop}[thm]{Proposition}

\newtheorem{clm}[thm]{Claim}
\newtheorem{cor}[thm]{Corollary}
\newtheorem{lemma}[thm]{Lemma}

\newtheorem{obs}[thm]{Observation}

\newtheorem{quest}[thm]{Question}

\theoremstyle{definition}

\title{Base partition for mixed families of finitary and cofinitary matroids}

\author{Joshua Erde}
\author{J.~Pascal Gollin}
\thanks{The second author was supported by the Institute for Basic Science (IBS-R029-C1).}
\author{Attila Jo\'{o}}
\thanks{The third author would like to thank the generous support of the Alexander von Humboldt Foundation and NKFIH OTKA-113047 and 129211}
\author{Paul Knappe}
\author{Max Pitz}
\address{Joshua Erde, Graz University of Technology, Institute of Discrete Mathematics, Steyrergasse 30, 8010 Graz, Austria}
\email{erde@math.tugraz.at}
\address{J.~Pascal Gollin, Discrete Mathematics Group, Institute for Basic Science (IBS), 55, Expo-ro, Yuseong-gu, Daejeon, Republic of Korea, 
34126}
\email{pascalgollin@ibs.re.kr}
\address{Attila Jo\'{o}, University of Hamburg, Department of Mathematics, Bundesstra{\ss}e 55 (Geomatikum), 20146 Hamburg, Germany}
\email{attila.joo@uni-hamburg.de}
\address{Paul Knappe, University of Hamburg, Department of Mathematics, Bundesstra{\ss}e 55 (Geomatikum), 20146 Hamburg, Germany}
\email{paul.knappe@studium.uni-hamburg.de}
\address{Max Pitz, University of Hamburg, Department of Mathematics, Bundesstra{\ss}e 55 (Geomatikum), 20146 Hamburg, Germany}
\email{max.pitz@uni-hamburg.de}


\begin{document}

\begin{abstract}
    Let ${\mathcal{M} = (M_i \colon i\in K)}$ be a finite or infinite family consisting of matroids 
    on a common ground set~$E$ each of which may be finitary or cofinitary. 
    We prove the following Cantor-Bernstein-type result: 
    If there is a collection of bases, one for each~$M_i$, 
    which covers the set~$E$, 
    and also a collection of bases which is pairwise disjoint, 
    then there is a collection of bases which partitions~$E$. 
    We also show that the failure of this Cantor-Bernstein-type statement for arbitrary matroid families is consistent relative to the axioms of set theory~ZFC. 
\end{abstract}

\maketitle

\section{Introduction}\label{sec:introduction}
\subsection{The main result}
Our starting point and main motivation was the following problem in infinite graph theory: 
Let ${G = (V,E)}$ be a connected graph. 
Given a cardinal~$\lambda$, 
a family ${(T_i \colon i < \lambda)}$ of spanning trees of~$G$ is called 
\begin{itemize}
    \item a \emph{$\lambda$-covering} of~$G$ if the union of the~${E(T_i)}$ is~$E$,
    \item a \emph{$\lambda$-packing} of~$G$ if the~$T_i$ are pairwise edge-disjoint, and
    \item a \emph{$\lambda$-partitioning} of~$G$ if~${(E(T_i) \colon i < \lambda)}$ is a partition of~$E$.
\end{itemize} 
Does the existence of a $\lambda$-covering and a $\lambda$-packing imply the existence of a $\lambda$-partitioning? 
Recently, we have proved that this holds when~$\lambda$ is infinite~\cite{EGJKP19}:

\begin{thm}
    \label{thm:intro-graphs-infinite}
    Let~$\lambda$ be an infinite cardinal. 
    Then a graph admits a $\lambda$-partitioning if and only if it admits both a $\lambda$-packing and a $\lambda$-covering.
\end{thm}

The point of departure for this paper is the natural question of whether the corresponding result also holds for finite~$\lambda$. 
Note that, when~$G$ is finite the result is clearly true, 
since the existence of a $\lambda$-covering implies that~$G$ has so few edges that any $\lambda$-packing must be a $\lambda$-partitioning. 
However, this argument fails when the edge set of~$G$ is infinite, even if~$\lambda$ is finite.

As is often the case when considering spanning trees in graphs, there is a natural generalisation of the problem to a question about bases in matroids, 
and indeed our proof was influenced by this change in view. 
Recall that for every connected graph ${G = (V,E)}$, finite or infinite, 
there is a matroid~${M(G)}$ with ground set~$E$ whose circuits are the subsets of~$E$ given by the cycles of~$G$, and whose bases are precisely the subsets of~$E$ given by spanning trees of~$G$. 
Every matroid of the form~${M(G)}$ is \emph{finitary}: 
it has the property that all of its circuits are finite. 
Recently, there has been a renewed interest in the theory of infinite matroids, after Bruhn, Diestel, Kriesell, Pendavingh and Wollan~\cite{BDKW13} gave a set of cryptomorphic axioms for infinite matroids which encompasses duality, 
generalising the usual independent set-, bases-, circuit-, closure- and rank-axioms for finite matroids, extending the work initiated by Higgs~\cites{MR274315} and later continued by Oxley~\cites{MR1165540}.
All the necessary background about infinite matroids needed for this paper will be given in Section~\ref{s:matroids}. 

Our earlier notions of packings, coverings and decompositions of graphs into spanning trees have natural generalisations for matroids and moreover, 
as will be key in our proof, for families of matroids on the same ground set. 

Let~${\mathcal{M} = (M_i \colon i \in K)}$ be a family of matroids on the 
same ground set~$E$.
A family ${(B_i \colon i \in K)}$ where~$B_i$ is a base of~$M_i$ for each~${i \in K}$ is called  
\begin{itemize}
    \item a \emph{base covering} of~$\mathcal{M}$ if~${\bigcup_{i \in K} B_i = E}$,
    \item a \emph{base packing} of~$\mathcal{M}$ if the~$B_i$ are pairwise disjoint,
    \item a \emph{base partitioning} of~$\mathcal{M}$ if ${(B_i \colon i \in K)}$ is a partition of~$E$.
\end{itemize} 

Before stating our main result we recall that a matroid is called \emph{cofinitary} if its dual is finitary, or equivalently, if all of its cocircuits are finite.
\begin{thm}
    \label{thm:main-thm-intro}
    Let~$\mathcal{M}$ be a family of matroids on a common ground set~$E$ each of which is either finitary or 
    cofinitary. 
    Then~$\mathcal{M}$ admits a base partitioning if and only if it admits both a base covering and a base packing. 
\end{thm}

By considering the special case of Theorem~\ref{thm:main-thm-intro} where ${K = \lambda}$ for some cardinal~$\lambda$ and each~$M_i$ is~${M(G)}$ for the same connected graph~$G$, 
we obtain the promised generalisation of Theorem~\ref{thm:intro-graphs-infinite} where~$\lambda$ is allowed to be finite. 

\begin{cor}
    Let~$\lambda$ be a cardinal. 
    Then a graph admits a $\lambda$-partitioning if and only if it admits both a $\lambda$-packing and a $\lambda$-covering.
\end{cor}

The case~${\left|K\right| = 2}$ has the following reformulation by taking the dual of one of the two matroids.

\begin{cor}\label{cor:share a base}
    Let~$M_i$ be a finitary or cofinitary matroid on the ground set~$E$ for~${i \in \{0,1\}}$. 
    If there are bases~$B_i$, $B_i'$ of~$M_i$ such that ${B_0 \subseteq B_1}$ and ${B_1' \subseteq B_0' }$, 
    then~$M_0$ and~$M_1$ share a base. 
    \qed
\end{cor}

Let us point out that the Cantor-Bernstein theorem 
mentioned in the abstract 
is a special case of our main result. 
Indeed, the graph theoretic reformulation of the Cantor-Bernstein theorem claims that if~${G = (V_0, V_1;E)}$ is a bipartite graph admitting a matching~${F_i \subseteq E}$ that covers~$V_i$ for~${i \in \{ 0,1 \}}$, then it admits a perfect matching. 
We describe the matchings in~$G$ as common independent sets of two partition matroids in the usual way, i.e., 
for~${i \in \{ 0,1 \}}$ let~$M_i$ be the finitary matroid on~$E$ whose circuits are the edge pairs with a common vertex in~$V_i$. 
Then~$F_i$ is a common independent set which is a base of~$M_i$ 
and therefore Corollary~\ref{cor:share a base} ensures the existence of a common base. 
Since~$G$ cannot contain isolated vertices, a common base is precisely a perfect matching.

However, perhaps surprisingly, the generalisation of Theorem~\ref{thm:main-thm-intro} (and even Corollary~\ref{cor:share a base}) to families of arbitrary matroids is consistently false.

\begin{thm}\label{thm:unprovable-intro}
    Assuming the Continuum Hypothesis, there is a countable matroid~$M$  
    such that the matroid family consisting of two copies of~$M$ admits a base packing and a base covering, but not a base partitioning.
\end{thm}

\subsection{Open questions}
Our results give rise to a number of open questions. 
Undoubtedly the most pressing one is whether Theorem~\ref{thm:main-thm-intro} extends to more general matroid families than finitary-cofinitary ones.
A matroid is called \emph{tame} if the intersection of any circuit and cocircuit is finite. 
Tame matroids are the largest subclass of matroids where usually ``nice behaviour'' is expected. 
For example the dual of a tame thin representable matroid is always thin representable which may fail for general matroids,  see~\cites{borujeni2015thin}. 
Even more special classes include the graphic matroids, tame matroids where every finite minor is graphic in the usual sense, see~\cite{bowler2018infinite}, 
and the $\Psi$-matroids, graphic matroids that are related to the Freudenthal compactification of locally finite graphs, 
see~\cite{bowler2013ubiquity}. 

\begin{quest}
    Is the analogue of Theorem~\ref{thm:main-thm-intro} true for tame matroids? 
    If not, does it hold for graphic matroids? 
    If not, does it hold for $\Psi$-matroids? 
\end{quest}

Another natural generalisation of Theorem~\ref{thm:main-thm-intro} concerns nearly (co-)finitary matroids which were first introduced in~\cite{MR3784779}. 
For a matroid~$M$, the \emph{finitarisation~$M^\textnormal{fin}$} of~$M$ is the matroid on the same ground set as~$M$ in which a set is independent if and only if all its finite subsets are independent in~$M$. 
We call a matroid~$M$ \emph{nearly finitary} if whenever~$B'$ is a base of~$M^\textnormal{fin}$ and~$B$ is a base of~$M$ with~${B' \supseteq B}$ then ${B' \setminus B}$ is finite. 
The definition of \emph{nearly cofinitary} is dual. 

\begin{quest}
    Is the analogue of Theorem~\ref{thm:main-thm-intro} true for families consisting only of nearly finitary and nearly cofinitary matroids? 
\end{quest}

Our last question is motivated by Theorem~\ref{thm:unprovable-intro}.

\begin{quest}
    Is there a counterexample in ZFC to the analogue of Theorem~\ref{thm:main-thm-intro} for families of arbitrary matroids? 
\end{quest}

\subsection{Structure of this paper}
The paper is structured as follows. 
In Section~\ref{s:matroids} we will give a short introduction to the theory of infinite matroids. 
Section~\ref{sec_prep} contains some auxiliary results which are used in our proof of Theorem~\ref{thm:main-thm-intro}, which appears in Section~\ref{sec_mainresult}. 
Finally, we end in Section~\ref{s:consistently-false} with the proof of Theorem~\ref{thm:unprovable-intro}.

\subsection{Sketch of the proof}
Let us give a brief sketch of the structure of the proof of Theorem~\ref{thm:main-thm-intro}.
We describe first our strategy in the case when every~$M_i$ is finitary and, furthermore,~$E = \{e_n: n\in \mathbb{N}\}$ is countable and~$K$ is finite. 
We will construct the desired partitioning ${(B_i \colon i \in K)}$ by recursion. 
We will maintain, at each step~$n$, a family ${(I_i^{n}\colon i \in K)}$ of subsets of~$E$ which can be extended to both a covering and a packing. 
More precisely, there exists a covering ${(R_i\colon i \in K)}$ such that ${R_i \supseteq I_i^{n}}$ for every~${i \in K}$ 
and there also exists a packing ${(P_i\colon i\in K)}$ with ${P_i \supseteq I_i^{n}}$. 
We call such a family feasible. 
During the recursion the sets~$I_i^{n}$ will be $\subseteq$-increasing in~$n$ for each~$i$, and we will ensure that   
${e_n \in \bigcup_{i\in K} I_i^{n+1}}$ and~$e_n$ is spanned by~$I_i^{n+1}$ in~$M_i$ for every~${i \in K}$. 
This will guarantee that for every~${i \in K}$, if we let~$B_i$ be the union of sets~$I_i^{n}$, then $B_i$ is a base of~$M_i$ and, moreover, the~$B_i$ partition~$E$. 

In order to extend our family ${(I_i^{n}\colon i \in K)}$ we will need to analyse how adding a single edge to one~$I_j^{n}$ and leaving the other~$I_i^{n}$ unchanged can effect the feasibility of the family. 
This is done in Subsection~\ref{subsec:tight sets} and leads to the notion of tight sets. 
Then, in Subsection~\ref{subsec: feasible extensions}, we build up the framework that we use in Subsection~\ref{subsec: the main lemmas} to prove that a general step of the recursion above can be done. 

Reducing the size of the matroid family is due to a short trick (see Claim~\ref{claim:wlog-lambda=3}). 
The reduction of the main result to countable matroids (via elementary submodel-type arguments) is more involved and fills most of Section~\ref{sec_mainresult}. 
Allowing non-finitary cofinitary matroids in the family makes some of the definitions and arguments more complicated, 
for example we need to build the bases~$B_i$ and their complements simultaneously during the recursion in contrast to the ``only finitary'' case we described.

\section{Infinite Matroids}
\label{s:matroids}

In this section we will gather some basic facts about infinite matroids that are necessary for this paper. 
Most of these facts are well-known for finite matroids. 
For a more detailed introduction to the theory of infinite matroids see~\cite{nathanhabil}.

A pair ${(E,\mathcal{I})}$ is a \emph{matroid} if ${\mathcal{I} \subseteq \mathcal{P}(E)}$ satisfies
\begin{enumerate}
    \item ${\emptyset \in \mathcal{I}}$; 
    \item $\mathcal{I}$ is downward closed; 
    \item For every ${I, B \in \mathcal{I}}$, where~$B$ is $\subseteq$-maximal in~$\mathcal{I}$ and~$I$ is not, there is an~${x \in B \setminus I}$ such that~${I+x \in \mathcal{I}}$; 
    \item For every~${X \subseteq E}$, every~${I \in \mathcal{I} \cap \mathcal{P}(X)}$ can be extended to a $\subseteq$-maximal element of~${\mathcal{I} \cap \mathcal{P}(X)}$. 
\end{enumerate}
The sets in~$\mathcal{I}$ are called \emph{independent} while the sets in ${\mathcal{P}(E) \setminus \mathcal{I}}$ are dependent. 
We note that, when~$E$ is finite, (1)--(4) are equivalent to the usual axiomatisation of matroids in terms of independent sets. 
The maximal independent sets are called \emph{bases} and the minimal dependent sets are called \emph{circuits}. 
Every dependent set contains a circuit (which, in fact, is non-trivial for infinite matroids). 
A matroid is called \emph{finitary} if all of its circuits are finite.

\begin{fact}\label{Fact-finitary}
    A matroid is finitary if and only if for every $\subseteq$-increasing chain of independent sets the union of the chain is also independent.   
\end{fact}

One example of a finitary matroid is the finite cycle matroid of an infinite graph. 
That is, if ${G = (V,E)}$ is a connected infinite graph, then we can define a matroid~${M(G) = (E,\mathcal{I})}$ by letting~${I \in \mathcal{I}}$ if and only if~$I$ does not contain any finite cycle\footnote{The so-called topological cycles of a graph can be infinite and define a matroid. In the finite cycle matroid only the finite such cycles (i.e., the usual graph theoretic cycles) are circuits.} of~$G$. 
It is straightforward to show that~${(E,\mathcal{I})}$ is an infinite matroid, and that the following facts are true:
\begin{itemize}
    \item A set of edges~${I \subseteq E}$ is independent if and only if it is a forest;
    \item A set of edges~$B$ is a base if and only if it is a spanning tree;
    \item A set of edges~$C$ is a circuit if and only if it is a finite cycle.
\end{itemize}

As with finite matroids, there is a notion of duality for infinite matroids. 
The \emph{dual} of a matroid~${M}$ is the matroid~${M^*}$ on the same edge set whose bases are the complements of the bases of~$M$. 
A matroid is called \emph{cofinitary} if its dual is finitary.

Given a matroid~${M = (E,\mathcal{I})}$ and an~${X \subseteq E}$, the \emph{restriction} of~$M$ to~$X$ is the matroid ${(X,\mathcal{I} \cap \mathcal{P}(X))}$ and it is denoted by~${M \restrict X}$.
For the restriction of~$M$ to~${E \setminus X}$ we also write ${M \setminus X}$ and call it the minor obtained by the \emph{deletion} of~$X$. 
We call the matroid ${(M^* \restrict X)^*}$ the minor of~$M$ obtained by the \emph{contraction} of~$X$ and denote it by~${M \slash X}$.
For the matroid obtained by the contraction of~${E \setminus X}$ we write~${M{.}X}$ and call it the \emph{contraction of~$M$ onto~$X$}. 
It is shown in~\cite{BDKW13} that all of these structures are indeed matroids and moreover, for every disjoint ${X, Y\subseteq E}$, 
${(M \slash X) \setminus Y = (M \setminus Y) \slash X}$. 
The \emph{minors} of~$M$ are the matroids of the form ${(M \slash X) \setminus Y}$ as above. 
We note that the class of finitary (cofinitary) matroids is closed under taking minors.

We also extend our notations to families of matroids. 
For a family~${\mathcal{M} = ( M_i \colon i \in K )}$ of matroids on the same ground set~$E$, we write 
${\mathcal{M} \setminus X}$, 
${\mathcal{M} \restrict X}$, 
${\mathcal{M} \slash X}$ or 
${\mathcal{M}{.}X}$, 
for the families 
${( M_i \setminus X \colon i \in K)}$, 
${( M_i \restrict X \colon i \in K)}$, 
${( M_i \slash X \colon i \in K)}$ or  
${( M_i {.} X \colon i \in K)}$ respectively.

If~$M_i$ is a matroid on~$E_i$ for~${i \in K}$, 
then the \emph{direct sum} ${\bigoplus_{i \in K} M_i}$ of the matroids~$M_i$ is the matroid on the disjoint union ${E := \bigdcup_{i \in K} E_i}$ of the sets~$E_i$ 
in which~${I \subseteq E}$ is independent if and only if~${I \cap E_i}$ is independent in~$M_i$ for each~${i \in K}$. 

For a matroid~${M = (E, \mathcal{I})}$ and~${X \subseteq E}$, 
we say the matroid~${M' := M \setminus X \oplus (X, \{\emptyset\})}$ is obtained from~$M$ by \emph{declaring the edges in~$X$ to be loops}. 
It is simple to check that this is a matroid, and we note in particular that each~${e \in X}$ is a loop in~$M'$. 

We say~${X \subseteq E}$ \emph{spans}~${e \in E}$ in matroid~$M$ if either~${e \in X}$ or there exists a circuit~${C \ni e}$ with~${C-e \subseteq X}$. 
We denote the set of edges spanned by~$X$ in~$M$ by~$\mathsf{span}_{M}(X)$. 
The operator~$\mathsf{span}_M$ is clearly extensive and increasing, and we note that it is also idempotent. 
An ${S \subseteq E}$ is \emph{spanning} in~$M$ if~${\mathsf{span}_{M}(S) = E}$. 

The following is the dual statement to Fact~\ref{Fact-finitary}.

\begin{fact}\label{Fact-cofinitary}
    A matroid is cofinitary if and only if for every $\subseteq$-decreasing chain of spanning sets the intersection of the chain is also spanning. 
\end{fact}

For a~${B \subseteq E}$ the following are equivalent:

\begin{itemize}
    \item $B$ is a maximal independent set,
    \item $B$ is a minimal spanning set,
    \item $B$ is an independent spanning set.
\end{itemize}

\begin{fact}
    \label{Fact-base}
    \cite{BDKW13}*{Lemma~3.7}
    If $B_0$, $B_1$ are bases of a matroid and ${\left|B_0\setminus B_1\right|<\aleph_0}$, then ${\left|B_0\setminus B_1\right|= 
\left|B_1\setminus B_0\right|}$.  
\end{fact}

\begin{fact}\label{Fact-unique-circ}
    If~$I$ is independent and spans~${e \notin I}$, then there is a unique circuit~${C(e,I)}$ contained in~${I+e}$. 
    For every~${f \in C(e,I)}$, the set~${I-f+e}$ is independent and spans the same set as~$I$.  
\end{fact}

\section{Preparatory Lemmas}
\label{sec_prep}

To allow ourselves some extra flexibility, we will broaden the concept for a base covering and a base packing slightly, and define covering and packing. 
Let~${\mathcal{M} = (M_i \colon i \in K)}$ be a family of arbitrary matroids (we will not restrict our scope to finitary-cofinitary families unless explicitly stated) on the same ground set~$E$. 
A family~${(R_i \colon i \in K)}$ of subsets of~$E$ is called a \emph{covering} of~$\mathcal{M}$ if~${\bigcup_{i \in K} R_i = E}$ and~$R_i$ is independent in~$M_i$ for all~${i \in K}$. 
Similarly, ${(P_i \colon i \in K)}$ is a \emph{packing} of~$\mathcal{M}$ if the~$P_i$ are pairwise disjoint and~$P_i$ is spanning in~$M_i$.
It is clear that there exists a covering (or packing, respectively) 
if and only if there is a base covering (or base packing, respectively). 

For brevity, when the family~$\mathcal{M}$ is clear from context we will simply write a \emph{covering of~${X \subseteq E}$} (or \emph{packing of~$X$}, respectively) to mean a covering (or packing, respectively) of the family~${\mathcal{M} \restrict X}$.

\subsection{Augmenting paths}
\label{subsec:aug-paths}
The algorithmic proof of the rank function formula for sums of finite matroids by Edmonds and Fulkerson (see~\cite{edmonds1965transversals}) is a variation of the so-called `augmenting path method' and has turned out to be an efficient tool in the theory of infinite matroids as well. 
The difficulty of the generalisation to infinite matroids lies mainly in finding the ``right'' formulation. 

Suppose that ${(I_i\colon i\in K)}$ is a family of pairwise disjoint sets such that~$I_i$ is independent in~$M_i$ for all~${i \in K}$ and suppose~${e \in E \setminus \bigcup_{i\in K} I_i}$. 
Roughly speaking, the next lemma tells that either 
there is another such family~${(J_i\colon i\in K)}$ of independent sets covering~${\bigcup_{i\in K}I_i+e}$ which is ``finitely close'' to the original family~${(I_i\colon i\in K)}$ 
or there is a ``witness'' for the non-existence of such sets~$J_i$. 
Note that in contrast to the case of finite matroids, this witness does not necessarily rule out the existence of a covering of~${\bigcup_{i\in K} I_i+e}$. 

\begin{lemma}\label{lem: augpath}
    Let ${(I_i\colon i \in K)}$ be a family of pairwise disjoint sets such that~$I_i$ is independent in~$M_i$ for all ${i \in K}$, and let~${e \in E \setminus \bigcup_{i\in K} I_i}$. 
    Then exactly one of the following two statements holds.
    \begin{enumerate}
        [label=(\arabic*)]
        \item \label{item: primal} There is a family of sets ${(J_i\colon i\in K)}$ and a ${k \in K}$ with the following properties:
        \begin{enumerate}
            [label=(\alph*)]
            \item $J_i$ is independent in~$M_i$, 
            \item ${J_i \cap J_j = \emptyset}$ for~${i \neq j \in K}$, 
            \item ${\bigcup_{i\in K} J_i = \bigcup_{i\in K} I_i + e}$, 
            \item ${\sum_{i\in K} \left|I_i \vartriangle J_i\right|<\aleph_0}$,  
            \item ${\mathsf{span}_{M_i}(J_i) = \mathsf{span}_{M_i}(I_i)}$ for ${i \neq k}$ and ${\mathsf{span}_{M_k} (J_k - f) = \mathsf{span}_{M_k} (I_k)}$ for \linebreak some~${f \in J_k}$. 
        \end{enumerate}
        \item \label{item: dual} There exists an ${X \subseteq \bigcup_{i\in K} I_i}$ for which ${I_i\cap X}$ spans~${X+e}$ in~$M_i$ for all~${i \in K}$. 
    \end{enumerate}
\end{lemma}

\begin{proof}
    Let us recall a well-known observation about exchanging multiple elements simultaneously in an independent set of a matroid.
    
    \begin{prop}\label{prop:simult change}
        Let~$I$ be independent in a fixed matroid, 
        let ${e_0, \dots, e_{n} \in \mathsf{span}(I) \setminus I}$ 
        and~${f_0, \dots, f_{n} \in I}$ 
        with ${f_m \in C(e_m, I)}$ 
        but ${f_m \notin C(e_\ell, I)}$ for~${\ell < m \leq n}$. 
        Then 
        \[
            {\left( I \cup \{e_0,\dots, e_{n}\} \right) \setminus \{f_0,\dots, f_{n}\}}
        \] 
        is independent and spans the same set as~$I$.
    \end{prop}
    
    \begin{proof}
        We use induction on~$n$. 
        The case~${n = 0}$ follows from Fact~\ref{Fact-unique-circ}. 
        Suppose that~${n > 0}$. 
        On the one hand, the set ${I-f_n+e_n}$ is independent and spans the same set as~$I$. 
        On the other hand, ${C(e_m, I-f_n+e_n) = C(e_m, I)}$ for~${m < n}$ because ${f_n \notin C(e_m, I)}$ for~${m < n}$. 
        Hence by using the induction hypothesis for~${I-f_n+e_n}$ and ${e_0,\dots, e_{n-1}, f_0,\dots, f_{n-1}}$ we are done.
    \end{proof}
    
    To show that at least one of~\ref{item: primal} and~\ref{item: dual} holds, we build an auxiliary digraph ${D = (V,A)}$ with ${V = E\cup K}$. 
    For each $i$ and each ${x \in E\setminus I_i}$ we have $(x,i) \in A$ if ${I_i + x}$ is independent in~$M_i$, and otherwise we have $(x,y) \in A$ for each $y \in {C_{M_i}(x, I_i) - x}$.
    
    We call a directed path from~$e$ to ${K \subseteq V}$ an \emph{augmenting path}.
    Suppose first that there is no augmenting path. 
    Let ${X'\subseteq E}$ be the set of vertices that are reachable from~$e$ in~$D$. 
    Clearly ${e \in X'}$ and by the construction of~$D$ if ${g \in X'}$ and ${i \in K}$ then either ${g \in I_i}$ or ${C_{M_i}(g, I_i)}$ is well-defined and a subset of~$X'$. 
    This implies that for ${X := X'-e}$ the set ${I_i \cap X}$ spans~${X+e}$ in~$M_i$ for~${i\in K}$.

    Assume now that there is an augmenting path ${x_0,\dots, x_{n+1}}$ in~$D$ where ${x_0 = e}$ and 
    ${x_{n+1} = k\in K}$. 
    By shortening the path we may assume that there is no ``jumping arc'',
    i.e., ${(x_{\ell}, x_m) \notin A}$ for~${1 \leq \ell+1 < m \leq n}$. 
    We define~$J_i'$ as the symmetric difference of~$I_i$ and \linebreak ${\{ x_m, x_{m+1}\colon x_{m+1}\in I_i \}}$. 
    The non-existence of jumping arcs ensures that Proposition~\ref{prop:simult change} is applicable for this simultaneous exchange at~$I_i$. 
    Hence~$J_i'$ is independent in~$M_i$ and spans the same set as~$I_i$. 
    For ${i \neq k}$ we let ${J_i := J_i'}$. 
    The arc ${(x_n, k) \in A}$ witnesses that ${I_k+x_n}$ is independent in~$M_k$ and hence so is ${J_k'+x_n =: J_k}$ because ${\mathsf{span}_{M_k}(I_k) = \mathsf{span}_{M_k}(J_k')}$. 
    Thus ${f := x_n}$ is suitable.

    Suppose for a contradiction that both~\ref{item: primal} and~\ref{item: dual} hold. 
    Then ${I_i^{ X } := I_i \cap X}$ is a base of ${M_i \restrict (X+e)}$. 
    The edge~$e$ witnesses that the independent sets ${J_i^{ X } := J_i\cap (X+e)}$ cover strictly more elements of~$X+e$ than the bases~${I_i^{X}}$. 
    Since $\sum_{i\in K} \left| I_i \vartriangle J_i\right|<\aleph_0$, by the pigeonhole principle there is some $ j\in K $ with ${\left|J_j^{X}\setminus I_j^{ X }\right|>\left|I_j^{ X }\setminus J_j^{ X }\right|}$ which contradicts Fact~\ref{Fact-base} after extending~$J_j^{X}$ to a base of ${M_j\restrict (X+e)}$ in an arbitrary way. 
\end{proof}

\subsection{Tight sets}
\label{subsec:tight sets}
A family of matroids~$\mathcal{M}$ is \emph{tight} if~$\mathcal{M}$ admits a covering and for every such covering~${(R_i \colon i \in K)}$ the set~$R_i$ is spanning in~$M_i$ for~${i \in K}$. 
To prove the tightness of a family~$\mathcal{M}$ it is clearly enough to consider only \emph{disjoint coverings}, i.e., coverings consisting of pairwise disjoint sets. 

\begin{prop}\label{prop: tight fulldef}
    Each covering of a tight matroid family~$\mathcal{M}$ is a partitioning of~$\mathcal{M}$.   
\end{prop}

\begin{proof}
    Let ${(R_i \colon i \in K)}$ be a covering of the tight matroid family~$\mathcal{M}$. 
    Then each~$R_i$ is independent and spanning in~$M_i$ which means it is a base of it. 
    If~${e \in R_i \cap R_j}$ for some~${e \in E}$ and~${i \neq j \in K}$ then by replacing~$R_i$ with~${R_i - e}$ we obtain a covering of~$\mathcal{M}$ 
    in which~${R_i - e}$ is not spanning in~$M_i$, contradicting the tightness of~$\mathcal{M}$. 
\end{proof}

A set~${X \subseteq E}$ is \emph{tight} with respect to~$\mathcal{M}$ (shortly $\mathcal{M}$-tight), if the family ${\mathcal{M} \restrict X}$ is tight.
Observe that the empty set is always tight.  

As mentioned, Lemma~\ref{lem: augpath} does not give a witness for the impossibility of covering the one extra edge~$e$. 
Using the idea of tight sets we can give such a witness.

\begin{lemma}\label{lem: one more cover}
    Let ${e \in E}$ and suppose that ${\mathcal{M}\restrict (E-e)}$ admits a covering. 
    Then~$\mathcal{M}$ admits a covering if and only if there is no $\mathcal{M}$-tight set~$X\subseteq E-e$ for which ${e \in \mathsf{span}_{M_i}(X)}$ for every~${i \in K}$. 
\end{lemma}

\begin{proof}
    Assume that~$X\subseteq E-e$ is an $\mathcal{M}$-tight set with ${e \in \mathsf{span}_{M_i}(X)}$ for every~${i \in K}$. 
    Suppose for a contradiction that $(R_i \colon i\in K)$ is a covering of~$E$ where ${e \in R_j}$ for some~${j \in K}$.
    Since ${(R_i\cap X \colon i\in K)}$ is a covering of the $\mathcal{M}$-tight set $X$, $X \subseteq \mathsf{span}_{M_j}(R_j\cap X)$ 
    and so, since $e\in \mathsf{span}_{M_j}(X) \subseteq \mathsf{span}_{M_j}(R_j\cap X)$ it follows that 
    $e \in \mathsf{span}_{M_j}(R_j\cap X)$. 
    However, $e \in R_j$, contradicting the independence of $R_j$ in $M_j$.
    
    To prove the converse, first we need  a terminology from~\cite{bowler2015matroid}. 
    We call a set~${W \subseteq E}$ a \emph{cowave} with respect to~$\mathcal{M}$ if~${\mathcal{M}{.}W}$ admits a covering. 
    Cowaves are closed under arbitrary large unions:
    indeed, let ${W := \bigcup_{\alpha < \kappa} W_\alpha}$ where~$W_\alpha$ is a cowave for~${\alpha < \kappa}$. 
    We fix a covering~${(R_{\alpha,i} \colon i \in K)}$ of~${\mathcal{M}{.}W_\alpha}$ for every~${\alpha < \kappa}$. 
    For each~${w \in W}$, let~$\alpha_w$ be the smallest ordinal with~${w \in W_{\alpha_w}}$. 
    We define ${R_i := \{ w \in W \colon w \in R_{\alpha_w,i} \}}$ for~${i \in K}$. 
    Then it is easy to see that ${(R_i \colon i \in K)}$ is a 
    covering of~${\mathcal{M}{.}W}$ 
    (for more details see~\cite{bowler2015matroid}*{Lemma 4.3 \& last line, p.~179}).
    
    Suppose that~$\mathcal{M}$ does not admit a covering. 
    Let~$W$ be the union of the cowaves not containing~$e$. 
    We show that ${X := (E-e) \setminus W}$ is an $\mathcal{M}$-tight set with ${e \in 
    \mathsf{span}_{M_i}(X)}$ for every~${i \in K}$, as desired.
    Note that there is no covering of~${(X+e)}$. 
    Indeed, if there were a covering ${(R'_i \colon i \in K)}$ of $(X+e)$ then, since there is a covering ${(Q_i \colon i \in K)}$ of~${\mathcal{M}{.}W}$ and~${W = E \setminus (X+e)}$, 
    it would follow that ${(R'_i \cup Q_i \colon i \in K)}$ is a covering of~$E$, 
    contradicting the assumption that~$\mathcal{M}$ has no covering.
    
    Since ${\mathcal{M} \restrict (E-e)}$ admits a covering, so does ${\mathcal{M} \restrict X}$. 
    Let ${(I_i \colon i\in K)}$ be an arbitrary disjoint covering of~$X$.
    We apply Lemma~\ref{lem: augpath} to ${(I_i \colon i \in K)}$ and~$e$, 
    and note that~\ref{item: primal} cannot occur, since there is no covering of~${(X+e)}$. 
    Hence we get some set ${X' \subseteq X}$ as in~\ref{item: dual}. 
    We claim that~${X' = X}$. 
    Observe that~${X \setminus X'}$ is a cowave in~${\mathcal{M} \restrict (X+e)}$ witnessed by the sets ${I_i \cap (X \setminus X')}$, since ${I_i \cap X'}$ spans ${X' + e}$ in~$M_i$ by Lemma~\ref{lem: augpath}. 
    A covering of~${\mathcal{M}{.}W}$ together with the sets $I_i \cap (X \setminus X')$ show that ${W \cup (X \setminus X')}$ is a cowave in~$\mathcal{M}$. 
    Since~$W$ was the largest cowave it follows that~${X = X'}$. 
    Therefore, ${I_i \cap X' = I_i \cap X = I_i}$ spans ${X' + e = X + e}$ in~$M_i$ for all~${i \in K}$. 
    Since ${(I_i \colon i \in K)}$ was an arbitrary disjoint covering of~$X$, the tightness of~$X$ follows. 
\end{proof}

Using this we can characterise when an edge~$e$ is never contained in~$R_j$, regardless of the covering~${(R_i \colon i\in K)}$ we consider.
\begin{cor}
    \label{cor:tight-char}
    Let ${\mathcal{M} = (M_i \colon i\in K)}$ be a family of matroids on the same ground set~$E$ admitting a covering and 
    let~${e \in E}$. 
    Then for all~${j\in K}$ the following statements are equivalent:
    \begin{enumerate}
        [label=(\roman*)]
        \item\label{item:char1} ${e \notin R_j}$ holds for every covering~${(R_i \colon i\in K)}$;
        \item\label{item:char2} there exists an $\mathcal{M}$-tight set~$X$ such that ${e \in \mathsf{span}_{M_j}(X) \setminus X}$.
    \end{enumerate}
\end{cor}

\begin{proof}
    Assume statement~\ref{item:char2} holds for some~${j \in K}$ and let ${(R_i \colon i\in K)}$ be a covering of~$\mathcal{M}$. 
    Since~$X$ is $\mathcal{M}$-tight, ${R_j \cap X}$ spans~$X$, and hence~$e$, in~$M_j$.
    However, since~$R_j$ is independent in~$M_j$ and ${e \notin X}$ we get that ${e \notin R_j}$, as desired.
    
    Now suppose that statement~\ref{item:char1} holds for some~${j \in K}$. 
    For~${i\neq j}$, let~$M_i'$ be the matroid that we obtain from~$M_i$ by declaring~$e$ to be a loop and let~${M_j' := M_j}$. 
    The resulting matroid family~$\mathcal{M}'$ does not admit a covering but~${\mathcal{M}\restrict(E-e)}$ does. 
    We apply Lemma~\ref{lem: one more cover} with~$\mathcal{M}'$ and~$e$ to obtain the desired~$X$.
\end{proof}

\begin{prop}
    \label{prop:tight-union}
    If~$\mathcal{M}$ is a family of matroids which admits a covering, 
    then the class of $\mathcal{M}$-tight sets is closed under arbitrary large unions and intersections.
\end{prop}

\begin{proof}
    Firstly, suppose that~$X_\alpha$ is an $\mathcal{M}$-tight set for ${\alpha<\kappa}$ and let ${(R_i \colon i\in K)}$ be an arbitrary covering of~${\bigcup_{\alpha<\kappa} X_\alpha}$. 
    Since ${(X_\alpha\cap R_i \colon i\in K)}$ is a covering of the $\mathcal{M}$-tight set~$X_\alpha$, we know that ${X_\alpha\cap R_i}$ 
    spans~$X_\alpha$ in~$M_i$ for~${i \in K}$. 
    But then~$R_i$ spans ${\bigcup_{\alpha<\kappa} X_\alpha}$ in~$M_i$ for~${i \in K}$ and hence $\bigcup_{\alpha<\kappa}X_\alpha$ is $\mathcal{M}$-tight.
    
    Secondly, suppose that ${(Q_i \colon i\in K)}$ is a covering of~${\bigcap_{\alpha<\kappa} X_\alpha}$. 
    Assume first that there is a covering ${( R_i \colon i \in K)}$ of~${\bigcup_{\alpha < \kappa} X_\alpha}$ 
    with~${Q_i \subseteq R_i}$. 
    Since~${\bigcup_{\alpha < \kappa} X_\alpha}$ is tight, $R_i$ spans ${\bigcup_{\alpha<\kappa} X_\alpha}$ for~${i \in K}$ 
    and the sets~$R_i$ are pairwise disjoint by Proposition~\ref{prop: tight fulldef}. 
    The latter implies that~${R_i \cap( \bigcap_{\alpha < \kappa} X_\alpha) = Q_i}$ for every~${i \in K}$. 
    Suppose that ${e \in (\bigcap_{\alpha < \kappa} X_\alpha) \setminus Q_j}$ for some~${j \in K}$ and let~${C := C_{M_j}(e,R_j)}$. 
    Since ${X_\alpha \ni e}$ is $\mathcal{M}$-tight for each~$\alpha$ and ${(R_i\cap X_\alpha \colon i\in K)}$ is a covering of it, we must have~${C-e\subseteq R_j\cap X_\alpha}$. 
    However, then ${C-e \subseteq R_j \cap(\bigcap_{\alpha < \kappa} X_\alpha) = Q_j}$. 
    It follows that~$Q_i$ spans ${\bigcap_{\alpha < \kappa} X_\alpha}$ in~$M_i$ for~${i \in K}$. 
    
    Finally we show that actually every covering of ${\bigcap_{\alpha < \kappa} X_\alpha}$ can be extended to a covering of~${\bigcup_{\alpha < \kappa} X_\alpha}$. 
    Indeed, since~$\mathcal{M}$ admits a covering there is at least one covering ${(Q_i \colon i\in K)}$ of ${\bigcap_{\alpha<\kappa} X_\alpha}$ which extends to a covering ${( R_i \colon i\in K)}$ of~${\bigcup_{\alpha<\kappa} X_\alpha}$.
    However then the sets ${R_i \setminus Q_i}$ form a covering of ${(\mathcal{M} \restrict \bigcup_{\alpha<\kappa}X_\alpha)/\bigcap_{\alpha<\kappa}X_\alpha}$ 
    because we have seen that~$Q_i$ spans ${\bigcap_{\alpha<\kappa} X_\alpha}$ in~$M_i$. 
    But then every covering ${(Q_i' \colon i\in K)}$ of ${\bigcap_{\alpha<\kappa} X_\alpha}$ can be extended to a covering of ${\bigcup_{\alpha<\kappa} X_\alpha}$ witnessed by the sets~${Q_i' \cup (R_i \setminus Q_i)}$ 
    which completes the proof. 
\end{proof}

\begin{obs}
    \label{obs:tight-cover extend}
    If~$\mathcal{M}$ is a family of matroids which admits a covering and~$X$ is $\mathcal{M}$-tight. 
    Then~${\mathcal{M}/X}$ admits a covering and hence every covering of~$X$ can be extended to a covering of~$\mathcal{M}$. 
\end{obs}

\subsection{Feasible extensions}
\label{subsec: feasible extensions}
In the proof of the main result Theorem~\ref{thm:main-thm-intro} we will construct the desired partitioning ${(B_i \colon i \in K)}$ recursively 
by keeping track, for each~$i$, of a set~$I_i$ of edges which will belong to~$B_i$, 
as well as a set ${S_i \supseteq I_i}$ such that edges not in~$S_i$ will not belong to~$B_i$ (briefly ${I_i \subseteq B_i \subseteq S_i}$). 
As we proceed, the set~$I_i$ will grow and the set~$S_i$ will shrink, and the base~$B_i$ will be sandwiched in the middle. 
Clearly a basic requirement, if this strategy is to succeed, is that~$I_i$ must be independent and~$S_i$ must be spanning in~$M_i$, otherwise~$B_i$ cannot be a base. 
However, if ${(B_i \colon i \in K)}$ is to be a partitioning, then these pairs must satisfy stronger conditions. 

Let us consider a family of ordered pairs 
${\mathcal{F} = (\left\langle I_i, S_i\right\rangle \colon i \in K)}$ where 
\begin{itemize}
    \item ${I_i\subseteq S_i \subseteq E}$; 
    \item $I_i$ is independent and~$S_i$ is spanning in~$M_i$;  
    \item ${I_i \cap I_j = \emptyset}$ for ${i \neq j \in K}$; and 
    \item ${\bigcup_{i \in K} S_i = E}$. 
\end{itemize}

We call such a family~$\mathcal{F}$ \emph{covering-feasible} with respect to~$\mathcal{M}$ 
if there exists a covering ${(R_i \colon i \in K)}$ of~$\mathcal{M}$ where~${I_i \subseteq R_i \subseteq S_i}$ for~${i \in K}$. 
We call such a covering ${\mathcal{F}}$-\emph{compatible}. 
The definition of \emph{packing-feasible} is analogous. 
Finally, $\mathcal{F}$ is \emph{feasible} if it is covering-feasible and packing-feasible. 
We say that~${\mathcal{F}' = (\left\langle I_i', S_i'\right\rangle \colon i \in K)}$ is an \emph{extension} of~$\mathcal{F}$ 
if ${I_i \subseteq I'_i \subseteq S'_i \subseteq S_i}$ 
holds for every~${i \in K}$. 

\begin{obs}\label{obs:S_i I_j}
    Every feasible family ${\mathcal{F} = (\left\langle I_i, S_i\right\rangle \colon i\in K)}$ admits a feasible extension 
    $\mathcal{F}' = (\left\langle I_i, S_i'\right\rangle \colon i\in K)$ for which ${S_i' \cap I_j = \emptyset}$ 
    whenever~${i \neq j}$. 
    Choosing ${S_i' := S_i \setminus \bigcup_{j \neq i} I_j}$ is appropriate. 
\end{obs}

We define ${\mathcal{M}(\mathcal{F})}$ to be the matroid family ${(M_i(\mathcal{F}) \colon i\in K)}$ 
where we obtain~${M_i(\mathcal{F})}$ from~$M_i$ by contracting~$I_i$, deleting~${\bigcup_{j\neq i} I_j}$ and declaring 
loops those remaining edges which are not in~$S_i$. 
Note that each~${M_i(\mathcal{F})}$ is a matroid on~${E \setminus \bigcup_{j \in K} I_j}$. 

\begin{obs}\label{obs:M(F) and F-comp}
    For a covering~${(R_i \colon i\in K)}$ of $\mathcal{M}(\mathcal{F})$, 
    the family~${( R_i \cup I_i \colon i\in K)}$ is an $\mathcal{F}$-compatible covering of~$\mathcal{M}$, 
    and, vice versa, 
    given an $\mathcal{F}$-compatible covering ${(R_i \colon i \in K)}$ of~$\mathcal{M}$ the family~${( R_i  \setminus\bigcup_{j \in K} I_j \colon i\in K)}$ is a covering of~${\mathcal{M}(\mathcal{F})}$. 
\end{obs}

Using our concept of tightness we can say when it is possible to extend a feasible~$\mathcal{F}$ by adding a 
particular edge~$e$ to~$I_j$. 

\begin{cor}
    \label{cor:cov-feasible}
    Let~${\mathcal{F} = (\left\langle I_i,S_i \right\rangle \colon i\in K)}$ be feasible
    and suppose that for~${e \in E \setminus \bigcup_{i \in K} I_i}$ the extension~$\mathcal{F}'$ that we obtain by adding~$e$ to~$I_j$ is not covering-feasible. 
    Then there is an~${\mathcal{M}(\mathcal{F})}$-tight set $ X\not \ni e $ that spans~$e$ in~$M_j(\mathcal{F})$. 
\end{cor}

\begin{proof}
    It follows from Observation~\ref{obs:M(F) and F-comp}, that there is no covering~${(R_i \colon i\in K)}$ 
    of~${\mathcal{M}(\mathcal{F})}$ with~${e \in R_j}$. 
    But then Corollary~\ref{cor:tight-char} implies that there is an $\mathcal{M}(\mathcal{F})$-tight set~$X$ spanning~$e$ 
    in~$M_j(\mathcal{F})$ as desired.  
\end{proof}

The next lemma allows us to ``eliminate'' all the non-empty tight sets by taking a suitable feasible extension.

\begin{lemma}
    \label{lem:elimin}
    Let ${\mathcal{F} = (\left\langle I_i,S_i\right\rangle \colon i\in K)}$ be feasible with respect to~$\mathcal{M}$, 
    let~$X$ be~${\mathcal{M}(\mathcal{F})}$-tight 
    and let ${(R_i \colon i\in K)}$ be a covering of~$X$ in~${\mathcal{M}(\mathcal{F})}$. 
    Then ${\mathcal{F}' = (\left\langle I_i',S_i\right\rangle \colon i\in K)}$ where ${I_i' = I_i \cup R_i}$ is a feasible extension of~$\mathcal{F}$ with respect to~$\mathcal{M}$ 
    and~$I_i'$ spans~${X \cap S_i}$ in~$M_i$ for every~${i \in K}$.
    
    Furthermore, if~$X$ is the $\subseteq$-largest ${\mathcal{M}(\mathcal{F})}$-tight set, then there is no non-empty~${\mathcal{M}(\mathcal{F}')}$-tight set.
\end{lemma}

\begin{proof}
    We first note that, since~$X$ is $\mathcal{M}(\mathcal{F})$-tight, it follows  that $R_i$ spans~$X$ in~${M_i(\mathcal{F})}$ and by Proposition~\ref{prop: tight fulldef} the $R_i$ are pairwise disjoint. 
    Hence by the definition of~${M_i(\mathcal{F}})$ we obtain that~$I_i'$ spans ${X \cap S_i}$ in~$M_i$. 
    
    To show the covering-feasibility of~$\mathcal{F}'$, it is enough to find a covering for~${\mathcal{M}(\mathcal{F}')}$ (see Observation~\ref{obs:M(F) and F-comp}). 
    Such a covering exists if and only if there is a covering 
    of~$\mathcal{M}(\mathcal{F})$ which extends~${(R_i \colon i\in K)}$. 
    Since~$X$ is $\mathcal{M}(\mathcal{F})$-tight, Observation~\ref{obs:tight-cover extend} guarantees a covering with this property. 
    
    We turn to the proof of the packing-feasibility of~$\mathcal{F}'$. 
    Since~$\mathcal{F}$ is packing-feasible, there is an $\mathcal{F}$-compatible packing~${(P_i \colon i\in K)}$. 
    Let ${P_i' := (P_i \setminus X) \cup I_i'}$ for~${i\in K}$. 
    Then~$P_i'$ is spanning in~$M_i$, since~$P_i$ is spanning in~$M_i$ and~$I_i'$ spans~${S_i \cap X \supseteq P_i \cap X}$ in~$M_i$. 
    Hence~$\mathcal{F}'$ is also packing-feasible, and so feasible.

    Finally, suppose that~$X$ is the $\subseteq$-largest ${\mathcal{M}(\mathcal{F})}$-tight set and~$Y$ is~${\mathcal{M}(\mathcal{F}')}$-tight. 
    We claim that ${X \cup Y}$ is~${\mathcal{M}(\mathcal{F})}$-tight. 
    Indeed, let ${(Q_i \colon i\in K)}$ be a covering of ${X \cup Y}$ in~$\mathcal{M}(\mathcal{F})$, 
    which exists by Observation~\ref{obs:M(F) and F-comp} and the feasibility of~$\mathcal{F}$. 
    As before, since~$X$ is $\mathcal{M}(\mathcal{F})$-tight, the set~${Q_i \cap X}$ spans~$X$ in~${M_i(\mathcal{F})}$. 
    Therefore ${(Q_i \setminus X \colon i \in K)}$ is a covering of~$Y$ in~${\mathcal{M}(\mathcal{F}')}$. 
    Then, since~$Y$ is $\mathcal{M}(\mathcal{F}')$-tight, the set ${Q_i \setminus X}$ spans~$Y$ in~${M_i(\mathcal{F}')}$. 
    Hence, $Q_i$ spans~${X \cup Y}$ in~${M_i(\mathcal{F})}$ and so ${X \cup Y}$ is $\mathcal{M}(\mathcal{F})$-tight. 
    Thus by the maximality of~$X$ we obtain~${Y = \emptyset}$. 
\end{proof}

\subsection{The main lemmas}
\label{subsec: the main lemmas}
Recall that our strategy will be to construct a partitioning as the `limit', in some sense, of a sequence of feasible families ${\mathcal{F} = ( \left\langle I_i, S_i \right\rangle \colon i \in K )}$. 
In order to guarantee that the family ${(B_i \colon i \in K)}$ that we construct is actually a partitioning, 
we will need to make sure that~$B_i$ is independent and spanning in each~$M_i$ and that the~$B_i$ cover~$E$ and are disjoint. 
The disjointness of the~$B_i$ will follow from the fact that the~$I_i$ are disjoint. 

In the case where~$M_i$ is finitary the independence of~$B_i$ in~$M_i$ will follow from Fact~\ref{Fact-finitary}. 
For the other two properties, we can think of them as a collection of `tasks': 
for each~${e \in E}$ we must make sure that~$e$ is contained in some~$I_i$ and that~$e$ is spanned in~$M_i$ by each~$I_i$. 
Hence the core of the proof will consist of the lemmas in this section, which will allow us to extend a feasible family to include an edge~$e$ in~${\bigcup I_j}$, 
or to span an edge~$e$ by a particular~$I_j$ if~$M_j$ is finitary. 

Conversely, in the case where~$M_i$ is cofinitary, the fact that~$B_i$ is spanning in~$M_i$ will follow from Fact~\ref{Fact-cofinitary}. 
However we will need to make sure `by hand' that~$B_i$ is independent in this case. 
However, since~$B_i$ is independent in~$M_i$ if and only if~${E \setminus B_i}$ is spanning in~$M_i^*$, we can also view this as a collection of `tasks': 
for each ${e \in E}$ we must make sure that~$e$ is spanned in~$M_i^*$ by~${E \setminus S_i}$. 
Hence we will need a further lemma to ensure that we can extend a feasible family to span an edge~$e$ by a particular~${E \setminus S_j}$ if~$M_j$ is cofinitary. 

These three lemmas will allow us to ensure that ${(B_i \colon i \in K)}$ is indeed a partitioning.

\begin{lemma}
    \label{lem:cover-e}
    Let ${\mathcal{F} = (\left\langle I_i, S_i\right\rangle \colon i \in K)}$ be feasible with respect to~$\mathcal{M}$ and 
    let~${e \in E}$. 
    Then there exists a feasible extension ${\mathcal{F}' = (\left\langle I_i', S_i\right\rangle \colon i\in K)}$ 
    of~$\mathcal{F}$ 
    such that ${e \in \bigcup_{i\in K} I_i'}$.
\end{lemma}

\begin{proof}
    By Lemma~\ref{lem:elimin}, we may assume that there is no non-empty~${\mathcal{M}(\mathcal{F})}$-tight set. 
    Let~${\mathcal{P} = (P_i \colon i\in K)}$ be an $\mathcal{F}$-compatible base packing. 
    We may assume that ${e \in \bigcup_{i\in K} P_i}$, as otherwise we take some ${j \in K}$ such that~$e$ is not an $M_j(\mathcal{F})$-loop (which exists by covering-feasibility and Observation~\ref{obs:M(F) and F-comp}) and 
    replace~$P_j$ by~${P_j + e - f}$ for some ${f \in 
    C_{M_j}(e,P_j) - e}$.
    
    Suppose that~${e \in P_j}$ and add~$e$ to~$I_j$   to obtain $ \mathcal{F}'$. 
    Then $\mathcal{P}$ is an $\mathcal{F}'$-compatible packing, and so $\mathcal{F}'$ is packing-feasible, 
    and by Corollary~\ref{cor:cov-feasible}, since there is no non-empty~${\mathcal{M}(\mathcal{F})}$-tight set and~$e$ is not an ${M_j(\mathcal{F})}$-loop, $\mathcal{F}'$ is also covering-feasible.
\end{proof}

\begin{lemma}
    \label{lem:span-e}
    Let ${\mathcal{F} = (\left\langle I_i, S_i\right\rangle \colon i\in K)}$ be feasible with respect to~$\mathcal{M}$, 
    let~${e \in E}$ and assume that~$M_j$ is finitary for some~${j \in K}$. 
    Then there exists a feasible extension ${\mathcal{F}' = (\left\langle I_i', S_i\right\rangle \colon i\in K)}$ of $\mathcal{F}$ such 
    that~$e$ is spanned by~$I_j'$ in~$M_j$. 
\end{lemma}

\begin{proof}
    We take a triple consisting of a feasible extension~$\mathcal{F}'$ of~$\mathcal{F}$, 
    an $\mathcal{F}'$-compatible packing ${\mathcal{P} = (P_i \colon i \in K)}$ 
    and an ${I \subseteq (P_j \setminus I_j')}$ with ${e \in \mathsf{span}_{M_j}(I\cup I'_j)}$ 
    in such a way that $\left|I\right|$ is as small as possible. 
    Since~$M_j$ is finitary, ${\left|I\right| =: n \in \mathbb{N}}$. If~${n = 0}$, then~$I_j'$ spans~$e$ in~$M_j$ and hence $\mathcal{F}'$ is as desired. 
    Suppose for a contradiction that~${n > 0}$ and let~${f \in I}$. 
    
    Adding~$f$ to~$I_j'$ results in a packing-feasible extension, since~$\mathcal{P}$ is a witness for this. 
    However, it cannot be covering-feasible since otherwise the triple it forms with~$\mathcal{P}$ and~${I-f}$ would contradict our choice. 
    
    Hence we may assume that adding~$f$ to~$I_j'$ results in an extension of $\mathcal{F}'$ which is not covering-feasible.
    We conclude by Corollary~\ref{cor:cov-feasible} that there is an ${\mathcal{M}(\mathcal{F}')}$-tight set~$X$ which spans~$f$ in~${M_j(\mathcal{F}')}$. 
    Let ${(R_i \colon i\in K)}$ be a covering of~$X$ in~${\mathcal{M}(\mathcal{F}')}$. 
    Applying Lemma~\ref{lem:elimin} to~$\mathcal{F}'$, $X$ and~${(R_i \colon i \in K)}$, 
    we get a feasible extension ${\mathcal{F}'' = (\left\langle I_i'', S_i \right\rangle \colon i \in K)}$ of~$\mathcal{F}'$ 
    such that~$I_i''$ spans~${X \cap S_i}$ in~$M_i$ for ${i \in K}$. 
    Hence, the system ${\mathcal{P}_X = ((P_i\setminus X) \cup I_i'' \colon i \in K)}$ is an $\mathcal{F}''$-compatible packing. 
    Let~$J$ be a maximal ${M_j/I_j''}$-independent subset of~${I \setminus X}$. 
    Then~${I_j'' \cup J}$ spans~$e$ in~$M_j$ because~${I_j' \cup I}$ does and~${I_j'' \supseteq I_j'}$ spans the edges~${X \cap S_j \supseteq X \cap I}$ in~$M_j$. 
    Furthermore, ${f \notin J}$ since the ${\mathcal{M}(\mathcal{F}')}$-tight~$X$ spans~$f$ in ${M_j(\mathcal{F}')}$ which implies by 
    ${f \in S_j}$ that~$X$ spans~$f$ in~${M_j/I_j'}$ and therefore so does~$I_j''$ in~$M_j$. 
    Hence, ${|J| < |I|}$ and therefore the feasible extension $\mathcal{F}''$ together with the $\mathcal{F}''$-compatible packing~$\mathcal{P}_X$ 
    and the set~$J$ contradicts the choice of~$\mathcal{F}'$, $\mathcal{P}$ and~$I$. 
\end{proof}

\begin{lemma}\label{lem:cospan-e}
    Let ${\mathcal{F} = (\left\langle I_i, S_i\right\rangle \colon i\in K)}$ be feasible with respect to~$\mathcal{M}$, let ${e \in E}$ and assume that~$M_j$ is cofinitary for some~${j \in K}$. 
    Then there exists a feasible extension \linebreak ${\mathcal{F}' = (\left\langle I_i, S_i'\right\rangle \colon i\in K)}$ of~$\mathcal{F}$ 
    such that~$e$ is spanned by~${E \setminus S_j'}$ in~$M^{*}_j$.
\end{lemma}

\begin{proof}
    We reduce the statement to Lemma~\ref{lem:span-e} using a dualisation argument. 
    Let ${\mathcal{M}^{*} = (M_i^{*} \colon i \in K)}$. 
    Assume first that~${\left|K\right|=2}$, say K=\{0,1\}. 
    Then ${(R_0, R_1)}$ is a covering of~$\mathcal{M}$ if and only if ${(E \setminus R_0, E \setminus R_1)}$ is a packing of~$\mathcal{M}^{*}$ and the analogue statement holds for packings. 
    Therefore an ${\mathcal{F} = (\left\langle I_i,S_i \right\rangle \colon i \in K)}$ is $\mathcal{M}$-feasible 
    if and only if ${\mathcal{F}^{*} := (\left\langle E \setminus S_i , E \setminus I_i \right\rangle \colon i \in K)}$ is $\mathcal{M}^{*}$-feasible. 
    Thus we can simply apply Lemma~\ref{lem:span-e} with~$\mathcal{F}^{*}$ and~$\mathcal{M}^{*}$, and ``dualising back'' the resulting extension.
    
    For ${\left|K\right| > 2}$ the argument is essentially the same except that we need to overcome some unpleasant technical difficulties. 
    Namely if ${(R_i \colon i \in K)}$ is a covering of~$\mathcal{M}$ then ${( E \setminus R_i \colon i \in K)}$ is usually 
    not a packing of~$\mathcal{M}^{*}$ when~${\left|K\right| > 2}$. 
    Indeed, instead of being pairwise disjoint, the sets ${E \setminus R_i}$ satisfy the weaker condition that each edge is 
    missing from at least one~${E \setminus R_i}$. 
    A similar problem occurs with the the dual object of packings. 
    To fix this, we use a technique from~\cite{bowler2015matroid}. 
    We define an auxiliary family 
    ${\widehat{\mathcal{M}} = (\widehat{M_i} \colon i\in K\cup \{ K \})}$ 
    of matroids on the common ground set~${E \times K}$. 
    For~${e \in E}$ and~${i \in K}$ the edge ${(e,j)}$ is a loop in~$\widehat{M_i}$ whenever~${i \neq j \in K}$. 
    A subset of~${E \times \{ i \}}$ is independent in~$\widehat{M_i}$ if and only if its projection to the first coordinate is independent in~$M_i^{*}$. 
    Finally, a set is defined to be independent in~$M_K$ if it meets ${\{ e \} \times K}$ in at most one element for every~${e \in E}$. 
    
    On the one hand, it follows directly from the definitions that if for an $\mathcal{M}$-feasible family
    ${\mathcal{F} = (\left\langle I_i , S_i \right\rangle \colon i \in K)}$ we take 
    ${\widehat{I}_i := (E \setminus S_i) \times \{ i \}}$, 
    ${\widehat{S}_i := (E \setminus I_i) \times \{ i \}}$, 
    ${\widehat{I}_K := \emptyset}$ and 
    ${\widehat{S}_K := E \times K}$,
    then ${\widehat{\mathcal{F}} := \left( \left\langle \widehat{I}_i , \widehat{S}_i \right\rangle \colon i \in K \cup \{ K \} \right)}$ is 
    $\widehat{\mathcal{M}}$-feasible. 
    
    On the other hand, if some ${\widehat{\mathcal{F}} = \left( \left\langle \widehat{I}_i,\widehat{S}_i\right\rangle \colon i \in K \cup \{ K\}\right)}$ is $\widehat{\mathcal{M}}$-feasible, then 
    \[
        {\left( \left\langle \mathsf{proj_E}\left[ (E\times \{ i \}) \setminus \widehat{S}_i\right] , \mathsf{proj_E}\left[ (E\times \{ i \}) 
        \setminus \widehat{I}_i\right] \right\rangle \colon i\in K \right)}
    \]
    is $\mathcal{M}$-feasible. 
    From this point the proof goes the same way as in the case~${\left|K\right| = 2}$ using~$\widehat{\mathcal{M}}$ instead of~$\mathcal{M}^{*}$. 
\end{proof}

\section{Proof of the main result}
\label{sec_mainresult}

The potentially uncountable size of~$K$ did not make any difference in the proofs so far but now it would cause some technical difficulties which motivates the following claim.
    
\begin{clm}
    \label{claim:wlog-lambda=3}
    Theorem~\ref{thm:main-thm-intro} is implied by its special case where~${\left|K\right| = 3}$.
\end{clm}

\begin{proof}
    Let~$\mathcal{M}$ be as in Theorem~\ref{thm:main-thm-intro}. 
    We build a new matroid family~${\mathcal{M}' = (M_0',M_1',M_2')}$ on the ground set~${E \times K}$ 
    such that~$\mathcal{M}'$ admits a packing/covering/partitioning 
    if and only if~$\mathcal{M}$ does. 
    Furthermore, $M_0'$ will be finitary 
    while~$M_1'$ and~$M_2'$ will be cofinitary. 
    
    First we take a copy~${\tilde{M_i}}$ of~$M_i$ on~${E \times\{ i \}}$ via the bijection ${e \mapsto (e,i)}$ for~${i\in K}$. 
    Let ${F := \{ i \in K \colon M_i \text{ is finitary} \}}$ and 
    let~$M_0'$ be the direct sum of the matroids~$\tilde{M_i}$ for~${i \in F}$ 
    together with the matroid~$(E \times (K \setminus F), \{ \emptyset\})$. 
    The construction of~$M_1'$ is analogous, we take the direct sum of the cofinitary copies together with~$(E \times F, \{\emptyset\})$. 
    Finally, we take~$M_2'$ to be the matroid whose set of circuits is ${\{ \{ e \} \times K \colon e \in E \}}$, or, in other words, 
    the direct sum of the duals of the $1$-uniform\footnote{A matroid is $1$-uniform if a set is independent if and only if it has at most~$1$ element.} matroids on~${\{ e \} \times K}$ for each~${e \in E}$. 
    It is easy to check directly from the definitions that~$\mathcal{M}'$ has the desired property. 
    Indeed, suppose for example that~${(R_i \colon i\in K )}$ is a covering of~$\mathcal{M}$. 
    Then we choose for each~${e \in E}$ an~${i_e \in K}$ such that~${e \in R_{i_e}}$. 
    Then ${I_0 := \bigcup_{i\in F} R_i \times \{ i \}}$ is independent in~$M_0'$, 
    ${I_1 := \bigcup_{i \in K \setminus F} R_i \times \{ i \}}$ is independent in~$M_1'$ and 
    ${I_2 := \bigcup_{e\in E} (\{ e \} \times K) \setminus \{ (e,i_e) \}}$ is independent in~$M_2'$. 
    Moreover, ${I_0 \cup I_1 \cup I_2 = E\times K}$. 
    The proof of the remaining implications are similarly easy and we leave them to the reader. 
\end{proof}

\begin{proof}
    [Proof of Theorem~\ref{thm:main-thm-intro}]
    By Claim~\ref{claim:wlog-lambda=3}, we may assume without loss of generality that~$K$ is finite. 
    Let ${F := \{ i \in K \colon M_i \text{ is finitary} \}}$ as above. 
    We apply transfinite induction on~${\left|E\right|}$. 
    If~$E$ is finite then every packing is a covering and every covering is a packing under the assumption that both exist. 
    Indeed, take a base packing. 
    Suppose for a contradiction that it is not a covering. 
    Then the sum of the ranks of the matroids is strictly less than~${\left| E \right|}$. 
    But then there is no covering, which is a contradiction. 
    
    Let ${\left|E\right| = \aleph_0}$ and let us fix an enumeration ${\{ e_n \colon n\in \mathbb{N} \}}$ of~$E$. 
    We build by recursion an increasing sequence ${(\mathcal{F}_n \colon n \in \mathbb{N})}$ 
    of feasible families where ${\mathcal{F}_{n} = (\left\langle I_i^{n}, S_i^{n}\right\rangle \colon i \in K)}$. 
    Let ${I_i^{0} = \emptyset}$ and ${S_i^{0} = E}$ for~${i \in K}$. 
    Note that ${(\left\langle I_i^0, S_i^0 \right\rangle \colon i \in K)}$ is feasible by the assumption that $\mathcal{M}$ admits a packing and a covering. 
    We demand that for every~${n \in \mathbb{N}}$ 
    \begin{enumerate}
        [label=(\arabic*)]
        \item\label{item:main-1-1} ${e_n \in\bigcup_{i\in K} I_i^{n+1} }$,
        \item\label{item:main-1-2} $e_n$ is spanned by~$I_i^{n+1}$ in~$M_i$ for~${i \in F}$, 
        \item\label{item:main-1-3} $e_n$ is spanned by~${E \setminus S_i^{n+1}}$ in~$M_i^{*}$ for~${i \in K\setminus F}$,
        \item\label{item:main-1-4} ${S_i^{n} \cap I_j^{n} = \emptyset}$ for~${j \neq i}$.
    \end{enumerate}
    Each step of the recursion can be done by applying in order Lemma~\ref{lem:cover-e} to arrange for~\ref{item:main-1-1}, Lemma~\ref{lem:span-e} to all $i \in F$ and Lemma~\ref{lem:cospan-e} to all $i \in K \setminus F$ to ensure \ref{item:main-1-2} and \ref{item:main-1-3} (this uses the finiteness of $K$), and Observation~\ref{obs:S_i I_j} for property \ref{item:main-1-4}. 
    Let ${B_i := \bigcup_{n=0}^{\infty} I_i^{n}}$ for each~${i \in K}$. 
    We claim ${(B_i \colon i\in K)}$ is the desired partitioning of~$\mathcal{M}$.
    
    Since the sets ${\{ I_i^{n} \colon i\in K \}}$ are pairwise disjoint for every~${n \in \mathbb{N}}$, so are the sets~$B_i$. 
    Also, by property~\ref{item:main-1-1}, $e_n$ is covered by some~$I_i^{n+1}$ and hence by~$B_i$ as well. 
    Therefore the sets~$B_i$ form a partition of~$E$. By property~\ref{item:main-1-4} it implies that ${B_i = \bigcap_{n=0}^{\infty} S_i^{n}}$ for each~${i \in K}$. 
    
    It remains to show that~$B_i$ is a base of~$M_i$ for each~${i \in K}$. 
    For~${i \in F}$, $B_i$ is independent by Fact~\ref{Fact-finitary} and by property~\ref{item:main-1-2}, $e_n$ is spanned by~$I_i^{n+1}\subseteq B_i$ in~$M_i$ for each~$n$. 
    Hence~$B_i$ is a basis of~$M_i$ for~${i \in F}$. 
    For~${i \in K \setminus F}$, ${B_i = \bigcap_{n=0}^{\infty} S_i^{n}}$ is spanning for~${i \in K\setminus F}$ by Fact~\ref{Fact-cofinitary}, 
    and by property~\ref{item:main-1-3} $e_n$ is spanned by~${E \setminus S_i^{n+1}\subseteq E\setminus B_i}$ in~$M_i^{*}$ for~${i \in K\setminus F}$. 
    Hence~$B_i$ is a basis of~$M_i$ for~${i \in K\setminus F}$. 
    This completes the proof of the countable case. 
    
    Suppose ${\kappa := \left|E\right| > \aleph_0}$. 
    Let ${(R_i \colon i \in K )}$ be a base covering of~$\mathcal{M}$ 
    and let ${(P_i \colon i \in K )}$ be a base packing. 
    We construct an increasing continuous chain 
    ${\left\langle E_\alpha \colon \alpha \leq \kappa \right\rangle}$ of subsets of~$E$ where 
    \begin{enumerate}
        [label=(\roman*)]
        \item\label{item:main-2-1} ${E_0 = \emptyset}$,
        \item\label{item:main-2-2} ${E_\kappa = E}$,
        \item\label{item:main-2-3} ${\left| E_\alpha \right| < \kappa}$ for~${\alpha < \kappa}$, 
        \item\label{item:main-2-4} ${E_\alpha \cap R_i}$ and ${E_\alpha \cap P_i}$ spans~$E_\alpha$ in~$M_i$ for~${i \in F}$, 
        \item\label{item:main-2-5} ${E_\alpha \cap (E \setminus R_i)}$ and ${E_\alpha \cap (E \setminus P_i)}$ 
        spans~${E_\alpha}$ in~$M_i^{*}$ for~${i \in K\setminus F}$. 
    \end{enumerate}
    
    Note that, since~$M_i$ is finitary for~${i \in F}$ and cofinitary for~${i \in K \setminus F}$ and since~$K$ is 
    finite, 
    restoring properties~\ref{item:main-2-4} and~\ref{item:main-2-5} after adding some~$e$ to an~$E_\alpha$ only requires adding some 
    countably many finite fundamental circuits and cocircuits to~$E_{\alpha}$, 
    thus the construction of such a chain above can be done by a straightforward transfinite recursion. 
    Note that the sets ${E^{\alpha} := E_{\alpha+1} \setminus E_\alpha\ (\alpha<\kappa)}$ form a partition of~$E$. 
    Let 
    \[
        M_i^{\alpha} := 
        \begin{cases} 
            (M_i\restrict E_{\alpha+1})/E_\alpha & \mbox{if } i \in F \\
            ((M_i^{*}\restrict E_{\alpha+1})/E_\alpha)^{*} & \mbox{if } i \in K \setminus F.
        \end{cases}
    \]
    
    \begin{lemma}
        \label{lem:smaller-problem}
        For every ${\alpha < \kappa}$ the matroid family
        ${\mathcal{M}_\alpha := (M_i^{\alpha} \colon i\in K)}$ on~$E^{\alpha}$ 
        admits a covering and a packing.
    \end{lemma}
    
    \begin{proof}
        Let ${R_i^{\alpha} := R_i\cap E^{\alpha}}$. 
        We show that ${(R_i^{\alpha} \colon i\in K)}$ is a covering of~$\mathcal{M}_{\alpha}$. 
        The non-trivial part of the statement is that~$R_i^{\alpha}$ is independent in~$M_{i}^{\alpha}$. 
        Property~\ref{item:main-2-4} ensures that for~${i \in F}$, the set ${R_i \cap E_\alpha}$ spans~$E_\alpha$ in~$M_i$, 
        and hence ${R_i \cap E^{\alpha}}$ remains independent after the contraction of~$E_\alpha$ in~$M_i$. 
        For~${i \in K\setminus F}$, we know by property~\ref{item:main-2-5} 
        that ${E_{\alpha+1} \setminus R_i}$ spans~$E_{\alpha+1}$ in~$M_i^{*}$. 
        Thus ${E^{\alpha} \setminus R_i}$ spans~$E^{\alpha}$ in~${M_i^{*}/E_\alpha}$ and hence in ${M_i^{*}/E_\alpha \restrict E_{\alpha+1}}$ as well. 
        By taking the dual it means that ${E^{\alpha} \cap R_i}$ is independent in~${M_i^{\alpha}}$. 
        The proof of the fact that~$\mathcal{M}_\alpha$ admits a packing is analogous. 
    \end{proof}
    
    By applying Lemma~\ref{lem:smaller-problem} and then the induction hypothesis for~$\mathcal{M}_\alpha$ (see property~\ref{item:main-2-3}), 
    we may fix a partition ${(B^{\alpha}_i \colon i\in K)}$ of~$E^{\alpha}$ where~$B_i^{\alpha}$ is a base of~$M_i^{\alpha}$. 
    Note that the set family ${\{ B_i^{\alpha} \colon \alpha < \kappa, i \in K \}}$ forms a partition of~$E$ 
    and hence so does the family~${(B_i \colon i\in K)}$ where ${B_i :=\bigcup_{\alpha<\kappa} B_i^{\alpha}}$. 
    It remains to show that~$B_i$ is a base of~$M_i$. 
    
    \begin{lemma}
        \label{lem:build-base}
        For ${i \in K}$ and ${\alpha \leq \kappa}$, ${B_{i,\alpha} := \bigcup_{\beta<\alpha} B_i^{\beta}}$ is a base of 
        $\begin{cases}
            M_i\restrict E_\alpha  &\mbox{if } i\in F \\
            M_i.E_\alpha & \mbox{if }i\in K \setminus F.  
        \end{cases}$
    \end{lemma}
    
    \begin{proof}
        Let~${i \in F}$. 
        By the construction, $ B_i^{\alpha} $ is a base of~${(M_i \restrict E_{\alpha+1})/E_{\alpha}}$. 
        We know that~${B_{i,\alpha}}$ is a base of~${M_i \restrict E_{\alpha}}$ by induction. 
        By combining these, we obtain that~$B_{i,\alpha+1}$ is a base of~${M_i \restrict E_{\alpha+1}}$. 
        At limit steps we obviously preserve the spanning property and~${i \in F}$ ensures that the independence as well. 
        For~${i \in K \setminus F}$, we use the reformulation that ${E_\alpha \setminus B_{i,\alpha}}$ is a base of ${M_i^{*}\restrict E_{\alpha}}$ for every~$\alpha$ which can be proved the same way. 
    \end{proof}
    
    Lemma~\ref{lem:build-base} tells for~${\alpha = \kappa}$ that~$B_i$ is a base of~$M_i$ which completes the proof.
\end{proof}

\section{A consistency result}
\label{s:consistently-false}

To prove our relative consistency result Theorem~\ref{thm:unprovable-intro}, we recall that a matroid is called \emph{uniform} if whenever~$I$ is independent with ${e \in I}$ and ${f \in E \setminus I}$, then~${I-e+f}$ is also 
independent. 
We also recall that the \emph{reaping number}~$\mathfrak{r}$ (also called the \emph{refinement number}) 
is the least cardinal such that there exists a family ${(A_i \colon i < \mathfrak{r})}$ of infinite subsets of~$\mathbb{N}$ 
for which there is no bipartition of~$\mathbb{N}$ that splits each~$A_i$ into two infinite pieces, see~\cite{vaughan}. 
Clearly ${\aleph_0 < \mathfrak{r} \leq 2^{\aleph_0}}$, and hence under the Continuum Hypothesis we have ${\mathfrak{r} = 2^{\aleph_0}}$. 
However, the same conclusion also holds for example under Martin's Axiom. 
On the other hand, also ${\mathfrak{r} < 2^{\aleph_0}}$ is consistently true, see~\cites{van1984integers,vaughan}. 

\begin{thm}\label{thm:unprovable-real}
    If ${\mathfrak{r} = 2^{\aleph_0}}$ then there is a uniform matroid~$U$ on~$\mathbb{N}$ 
    such that the matroid family consisting of two copies of~$U$ admits a packing and a covering, but not a partitioning. 
\end{thm}

This section is devoted to the proof of Theorem~\ref{thm:unprovable-real}. 
It is based on a construction of Bowler and Geschke in~\cite{bowler2016self}. 
Among other results they showed that the existence of a matroid admitting two bases with different infinite sizes is consistent with ZFC (and actually independent of ZFC). 

\begin{proof}[Proof of Theorem~\ref{thm:unprovable-real}]
    Let us denote the set of the infinite and (simultaneously) co-infinite subsets of a set~$X$ by~$\mathcal{P}^{*}(X)$ and we write ${X \subseteq^{*} Y}$ if ${\left| X \setminus Y \right| < \aleph_0}$. 
    We construct a ${\mathcal{B} \subseteq \mathcal{P}^{*}(\mathbb{N})}$ which is a set of the bases of a desired~$U$. 
    Conditions~\ref{item:unprov-1}--\ref{item:unprov-3} in the following list are expressing that~$\mathcal{B}$ is the set of the bases of a uniform matroid on~$\mathbb{N}$ (see~\cite{bowler2016self}), 
    whereas~\ref{item:unprov-4} and~\ref{item:unprov-5} guarantee that~$U$ admits a packing and a covering but not a partitioning. 
    \begin{enumerate}
        [label=(\arabic*)]
        \item \label{item:unprov-1} The elements of~$\mathcal{B}$ are pairwise $\subseteq$-incomparable. 
        \item \label{item:unprov-2} Whenever ${e \in B \in \mathcal{B}}$ and ${f \in \mathbb{N} \setminus B}$, then~${B-e+f \in \mathcal{B}}$. 
        \item\label{item:unprov-3} For every ${I \subseteq X \subseteq \mathbb{N}}$, there is a ${B \in \mathcal{B}}$ such that one of the following holds:
            \begin{enumerate}
                [label=(\alph*)]
                \item\label{item:unprov-3a} ${B \subseteq I}$
                \item\label{item:unprov-3b} ${I \subseteq B \subseteq X}$
                \item\label{item:unprov-3c} ${X \subseteq B}$
            \end{enumerate}
        \item\label{item:unprov-4} There are no disjoint ${B, B' \in \mathcal{B}}$ with~${B \cup B' = \mathbb{N}}$. 
        \item\label{item:unprov-5} There are ${R_0, R_1, P_0, P_1 \in \mathcal{B}}$ with~${R_0 \cup R_1 = \mathbb{N}}$ and~${P_0 \cap P_1 = \emptyset}$. 
    \end{enumerate}
    
    Let us fix a well-order~$\prec$ of the set ${\mathcal{A} := \{ (I,X) \colon I\subseteq X\subseteq \mathbb{N} \}}$ of type~$2^{\aleph_0}$ in which ${(\emptyset, I) \preceq (I,X)}$ for every~${I \subseteq X \subseteq \mathbb{N}}$. 
    Let ${\{ (I_\alpha,X_\alpha) \colon \alpha<2^{\aleph_0} \}}$ be the enumeration of~$\mathcal{A}$ given by~$\prec$. 
    We build the family~$\mathcal{B}$ by transfinite recursion as the union of an increasing continuous chain ${\left\langle \mathcal{B}_\alpha \colon \alpha<2^{\aleph_0} \right\rangle}$ with 
    ${\left|\mathcal{B}_{0} \right|, \left|\mathcal{B}_{\alpha+1}\setminus \mathcal{B}_\alpha \right| \leq \aleph_0}$ for $\alpha<2^{\aleph_0}$ 
    where~$\mathcal{B}_\alpha$ satisfies~\ref{item:unprov-1}, \ref{item:unprov-2}, \ref{item:unprov-4}, \ref{item:unprov-5} 
    and the restriction of~\ref{item:unprov-3} to the pairs ${\{ (I_\beta,X_\beta) \colon \beta<\alpha \}}$ that we call~\ref{item:unprov-3}($\alpha$). 
    Let ${R_0, R_1 \subseteq \mathbb{N}}$ be such that they cover~$\mathbb{N}$ and all the sets ${R_0 \setminus R_1}$, ${R_1\setminus R_0}$, ${R_0\cap R_1}$ are infinite. 
    We also pick disjoint sets ${P_0, P_1 \subseteq \mathbb{N}}$ such that all of ${P_0}, {P_1}, {\mathbb{N} \setminus (P_0 \cup P_1)}$ have an infinite intersection with any of ${R_0 \setminus R_1}$, ${R_1\setminus R_0}$, ${R_0\cap R_1}$. 
    By defining~$\mathcal{B}_0$ as the closure of the set ${\{ R_0, R_1, P_0, P_1 \}}$ under~\ref{item:unprov-2}, we get neither $\subseteq$-comparable sets nor two sets forming a partitioning. 
    Clearly the conditions cannot be ruined at limit steps. 
    Suppose that~$\mathcal{B}_\alpha$ is defined. 
    If~\ref{item:unprov-3}($\alpha+1$) is satisfied by~$\mathcal{B}_\alpha$ then let ${\mathcal{B}_{\alpha+1} := \mathcal{B}_\alpha}$. 
    Suppose that~\ref{item:unprov-3}($\alpha+1$) is not satisfied.
    
    \begin{clm}
        $\left| X_\alpha \setminus I_\alpha \right| = \aleph_0$
    \end{clm} 
    
    \begin{proof}
        By the choice of~$\prec$, we know that ${(\emptyset, I_\alpha) \preceq(I_\alpha, X_\alpha)}$. 
        The inequality must be strict otherwise ${I_\alpha = X_\alpha = \emptyset}$ and hence~$\mathcal{B}_\alpha$ satisfies~\ref{item:unprov-3}($\alpha+1$) which is a contradiction. 
        By the induction hypothesis, the condition corresponding to ${(\emptyset, I_\alpha)}$ is satisfied in~$\mathcal{B}_\alpha$. 
        Since $\mathcal{B}_\alpha$ does not satisfy property~\ref{item:unprov-3}($\alpha+1$), there is no ${B \in \mathcal{B}_\alpha}$ with ${B \subseteq I_\alpha}$. 
        Therefore, we must have a ${B \in \mathcal{B}_\alpha}$ with~${B \supseteq I_\alpha}$. 
        If ${X_\alpha \setminus I_\alpha}$ were finite, then property~\ref{item:unprov-2} and~$B$ would give a ${B' \in \mathcal{B}_\alpha}$ with either ${B'\supseteq X_\alpha}$ or ${I_\alpha \subseteq B' \subseteq X_\alpha}$, both of which  contradict the assumption that~$\mathcal{B}_\alpha$ does not satisfy~\ref{item:unprov-3}($\alpha+1$). 
    \end{proof} 
    
    Consider the set
    \[
        \mathcal{F} = 
        \{ B \cap (X_\alpha\setminus I_\alpha) \colon B \in \mathcal{B}_\alpha \} \cup \{ ( X_\alpha\setminus I_\alpha) \setminus B \colon B \in \mathcal{B}_\alpha \}. 
    \]
    Since $\left|\mathcal{F}\right|\leq  \left|\alpha\right|+ \aleph_0  <2^{\aleph_0}$ and $\mathfrak{r}=2^{\aleph_0}$,  there is a ${G \in 
    \mathcal{P}^{*}(X_\alpha \setminus I_\alpha)}$ such that both~$G$ and ${ (X_\alpha \setminus I_\alpha) \setminus G}$ 
    have an infinite intersection with each infinite element of~$\mathcal{F}$. 
    Let ${B_\alpha := I_\alpha \cup G}$ and we extend~$\mathcal{B}_\alpha$ with all the sets 
    \[
        \{ Y \subseteq \mathbb{N} \colon \left| Y \setminus B_\alpha \right| = \left| B_\alpha \setminus Y \right| < \aleph_0 \}
    \] 
    to obtain $\mathcal{B}_{\alpha+1}$. 
    Properties~\ref{item:unprov-2} and~\ref{item:unprov-3}($\alpha+1$) obviously hold for~$\mathcal{B}_{\alpha+1}$. 
    To show~\ref{item:unprov-1}, it is enough to prove that there is no~${B \in \mathcal{B}_\alpha}$ for which ${B \subseteq^{*} B_\alpha}$ or ${B_\alpha \subseteq^{*} B}$. 
    Suppose for a contradiction that ${B \subseteq^{*} B_\alpha}$ for some ${B \in \mathcal{B}_\alpha}$. 
    Then ${B \cap (X_\alpha \setminus I_\alpha)}$ must be finite since otherwise it contains infinitely many elements not contained by~$G$. 
    But then ${B \subseteq^{*} I_\alpha}$ and by applying~\ref{item:unprov-2} we obtain a ${B' \in \mathcal{B}_{\alpha}}$ for which either ${B' \subseteq I_\alpha}$ or ${I_\alpha \subseteq B' \subseteq X_\alpha}$, 
    thus property~\ref{item:unprov-3}($\alpha+1$) was satisfied in~$\mathcal{B}_\alpha $, a contradiction. 
    Ruling out the existence of a~${B \in \mathcal{B}_\alpha}$ with ${B_\alpha \subseteq^{*} B}$ is analogous. 
    
    To check~\ref{item:unprov-4}, take an arbitrary~${B \in \mathcal{B}_\alpha}$. 
    If ${B \cap (X_\alpha\setminus I_\alpha)}$ is infinite then the choice of~$G$ guarantees that ${B \cap G}$ is infinite and therefore $B\cap B_\alpha$ as well. 
    If ${B \cap (X_\alpha\setminus I_\alpha)}$ is finite, then ${(X_\alpha \setminus I_\alpha})\setminus B$ is infinite 
    and so is its intersection with ${(X_\alpha \setminus I_\alpha)\setminus G}$ by the choice of~$G$. 
    Therefore ${\mathbb{N}\setminus (B \cup B_\alpha)}$ is infinite. 
\end{proof}

\bibliographystyle{unsrtnat}

\begin{bibdiv}
\begin{biblist}
\bib{MR3784779}{article}{
   author={Aigner-Horev, Elad},
   author={Carmesin, Johannes},
   author={Fr\"{o}hlich, Jan-Oliver},
   title={On the intersection of infinite matroids},
   journal={Discrete Math.},
   volume={341},
   date={2018},
   number={6},
   pages={1582--1596},
   
}	

\bib{borujeni2015thin}{article}{
      author={Borujeni, S Hadi~Afzali},
      author={Bowler, Nathan},
       title={Thin sums matroids and duality},
        date={2015},
     journal={Advances in Mathematics},
      volume={271},
       pages={1\ndash 29},
}

\bib{nathanhabil}{thesis}{
      author={Bowler, Nathan},
       title={Infinite matroids},
        type={Habilitation thesis, University Hamburg},
        date={2014},
}

\bib{bowler2013ubiquity}{article}{
      author={Bowler, Nathan},
      author={Carmesin, Johannes},
       title={The ubiquity of psi-matroids},
        date={2013},
     journal={arXiv preprint arXiv:1304.6973},
}

\bib{bowler2015matroid}{article}{
      author={Bowler, Nathan},
      author={Carmesin, Johannes},
       title={Matroid intersection, base packing and base covering for infinite
  matroids},
        date={2015},
     journal={Combinatorica},
      volume={35},
      number={2},
       pages={153\ndash 180},
}

\bib{bowler2018infinite}{article}{
      author={Bowler, Nathan},
      author={Carmesin, Johannes},
      author={Christian, Robin},
       title={Infinite graphic matroids},
        date={2018},
     journal={Combinatorica},
      volume={38},
      number={2},
       pages={305\ndash 339},
}

\bib{bowler2016self}{article}{
      author={Bowler, Nathan},
      author={Geschke, Stefan},
       title={Self-dual uniform matroids on infinite sets},
        date={2016},
     journal={Proceedings of the American Mathematical Society},
      volume={144},
      number={2},
       pages={459\ndash 471},
}

\bib{BDKW13}{article}{
      author={Bruhn, H.},
      author={Diestel, R.},
      author={Kriesell, M.},
      author={Pendavingh, R.},
      author={Wollan, P.},
       title={Axioms for infinite matroids},
        date={2013},
     journal={Advances in Mathematics},
      volume={239},
       pages={18\ndash 46},
}

\bib{van1984integers}{incollection}{
      author={Douwen, Eric K.~{van}},
       title={The integers and topology},
        date={1984},
   booktitle={Handbook of set-theoretic topology},
   publisher={Elsevier},
       pages={111\ndash 167},
}

\bib{edmonds1965transversals}{article}{
      author={Edmonds, Jack},
      author={Fulkerson, D.~R.},
       title={Transversals and matroid partition},
        date={1965},
        ISSN={0160-1741},
     journal={J. Res. Nat. Bur. Standards Sect. B},
      volume={69B},
       pages={147\ndash 153},
}

\bib{EGJKP19}{unpublished}{
      author={Erde, J.},
      author={Gollin, P.},
      author={Jo\'{o}, A.},
      author={Knappe, P.},
      author={Pitz, M.},
       title={A {C}antor-{B}ernstein-type theorem for spanning trees in
  infinite graphs},
        date={2019},
        note={https://arxiv.org/abs/1907.09338},
}

\bib{MR274315}{article}{
      author={Higgs, D.~A.},
       title={Matroids and duality},
        date={1969},
        ISSN={0010-1354},
     journal={Colloq. Math.},
      volume={20},
       pages={215\ndash 220},
         url={https://doi.org/10.4064/cm-20-2-215-220},
}

\bib{MR1165540}{incollection}{
      author={Oxley, James},
       title={Infinite matroids},
        date={1992},
   booktitle={Matroid applications},
      series={Encyclopedia Math. Appl.},
      volume={40},
   publisher={Cambridge Univ. Press, Cambridge},
       pages={73\ndash 90},
         url={https://doi.org/10.1017/CBO9780511662041.004},
}

\bib{vaughan}{incollection}{
      author={Vaughan, Jerry~E.},
       title={Small uncountable cardinals and topology},
        date={1990},
   booktitle={Open problems in topology},
   publisher={North-Holland, Amsterdam},
       pages={195\ndash 218},
}

\end{biblist}
\end{bibdiv}

\end{document}